\documentclass[10pt]{amsart}
\usepackage{amsmath}
\usepackage{amssymb}
\usepackage{relsize}
\def\shu{\joinrel{\!\scriptstyle\amalg\hskip -3.1pt\amalg}\,}

\newcommand{\un}{\hbox{\bf 1}}

\def\pack{{\rm pack}}
\def\cal{\mathcal}

\def\Sym{{\bf Sym}}     
\def\NCSF{{\bf Sym}}    

\def\Surj{{\mathrm{Sj}}}
\def\SURJ{{\mathbf{Sj}}}

\newtheorem{example}{Example}[section]

\newtheorem{theorem}[example]{Theorem}

\newtheorem{definition}[example]{Definition}
\newtheorem{proposition}[example]{Proposition}

\newtheorem{lemma}[example]{Lemma}

\def\up#1{\raise 1ex\hbox{\footnotesize#1}}

\def\U0{{\cal U}_0(gl_N)}

\def\<{\langle}
\def\>{\rangle}

\def\shuff#1#2{\mathbin{
      \hbox{\vbox{
        \hbox{\vrule
              \hskip#2
              \vrule height#1 width 0pt
               }%
        \hrule}%
             \vbox{
        \hbox{\vrule
              \hskip#2
              \vrule height#1 width 0pt
               \vrule }%
        \hrule}%
}}}

\def\shuffl{{\mathchoice{\shuff{7pt}{3.5pt}}%
                        {\shuff{6pt}{3pt}}%
                        {\shuff{4pt}{2pt}}%
                        {\shuff{3pt}{1.5pt}}}}%

\def\ash{\, \underline{\shuffl} \, }

\def\DessinMatrix#1{\vcenter{\hbox{\makebox[1.7ex]{$#1$}}}}
\def\GenMatrix#1{\vcenter{\halign{&$\DessinMatrix{##}$\cr#1}}\egroup}

\def\setinterlineskip#1{\baselineskip=0pt
  \lineskip=#1 \lineskiplimit=\maxdimen}

\def\matrice{%
  \bgroup
  \let\ =\omit
  \let\\=\cr
  \setinterlineskip{4.0pt}
  \GenMatrix}

\def\DessinsMatrix#1{\vcenter{\hbox{\makebox[1.3ex]{$\scriptstyle#1$}}}}
\def\GensMatrix#1{\vcenter{\halign{&$\DessinsMatrix{##}$\cr#1}}\egroup}

\def\smallmatrice{%
  \bgroup
  \let\ =\omit
  \let\\=\cr
  \setinterlineskip{3.0pt}
  \GensMatrix}

\newlength{\Hackl}
\newcommand{\Hack}{\vrule height \Hackl width 0pt}
\newcommand{\indexmat}%
    {\smallmatrice{\Hack a\\\Hack b\\\Hack c\\\Hack d\\\Hack e\\}}

\newdimen\Squaresize \Squaresize=14pt
\newdimen\Thickness \Thickness=0.5pt

\def\Square#1{\hbox{\vrule width \Thickness
   \vbox to \Squaresize{\hrule height \Thickness\vss
      \hbox to \Squaresize{\hss#1\hss}
   \vss\hrule height\Thickness}
\unskip\vrule width \Thickness}
\kern-\Thickness}

\def\Vsquare#1{\vbox{\Square{$#1$}}\kern-\Thickness}

\title[Exponential Lie series]%
{Flows and stochastic Taylor series\\ in It\^o calculus}

\author[Ebrahimi-Fard,  Malham, Patras and Wiese]{
Kurusch Ebrahimi--Fard$^\MakeLowercase{a}$, 
Simon J.A.~Malham$^\MakeLowercase{b}$,\\  
Fr\'ed\'eric Patras$^\MakeLowercase{c}$
and Anke Wiese$^\MakeLowercase{b}$}

\address[a]{
ICMAT, 
CSIC-UAM-UCM-UC3M, 
C/ Nicol\'as Cabrera 13-15, 28049 Madrid, Spain. 
On leave from UHA, Mulhouse, France.
}

\address[b]{Maxwell Institute for Mathematical Sciences
and School of Mathematical and Computer Sciences,
Heriot-Watt University, Edinburgh EH14 4AS, UK.
}

\address[c]{
Labo.~de Math\'ematiques J.A.~Dieudonn\'e,
Universit\'e de Nice - Sophia Antipolis,
Parc Valrose, 06108 Nice Cedex 02, France.
}

\date{\today}

\begin{document}

\begin{abstract}
For stochastic systems driven by continuous semimartingales an explicit formula for the logarithm of the  It\^o flow map is given. A similar formula is also obtained for solutions of linear matrix-valued SDEs driven by arbitrary semimartingales. The computation relies on the lift to quasi-shuffle algebras of formulas involving products of It\^o integrals of semimartingales. Whereas the Chen--Strichartz formula computing the logarithm of the Stratonovich flow map is classically expanded as a formal sum indexed by permutations, the analogous formula in It\^o calculus is naturally indexed by surjections. This reflects the change of algebraic background involved in the transition between the two integration theories. 
\end{abstract}

\maketitle



\section{Introduction}
\label{sect:intro}

The setting of our work is classical stochastic calculus, as exposed, e.g., in the classical textbooks \cite{Jacod2002,Protter}. Its aim is to obtain an explicit formula for the logarithm of the It{\^o} flow map associated to a stochastic differential system, generalizing the classical work of Ben Arous \cite{Benarous} on Stratonovich flows and stochastic Taylor series. 

In this introduction, we briefly state the main result of this work. A description of the historical background and details on the underlying definitions and objects of study are postponed to the next section.

Let  $\{X^1,\, X^2, \ldots , X^N\}$ be scalar continuous semimartingales. We assume, without loss of generality, that $X^i_0=0$ and that their quadratic covariation, or square bracket operation, is such that $[X^i, X^j]\equiv 0$ for $i \not= j$. We consider the general stochastic differential system (in what follows the notation $\int_0^t \cdots \mathrm{d}X^i_s$ refers to It\^o integrals)
\begin{equation}\label{stochSys}
	Y_t  = \, Y_0 + \sum_{i=1}^N \int_0^tV_i(Y_s)\,\mathrm{d}X^i_s,
\end{equation}
where  $V_i: \mathbb R^d \to \mathbb R^d$ are smooth vector fields. In the following, we will identify $V_i$ with the partial differential operator $V_i\partial_y:=\sum_{j=1}^d V_i^j\partial_{y_j}$. For $i=N+1, \ldots,\, 2N$, let $X^i$ denote the quadratic variation of $X^{i-N}$, that is $X^i=[X^{i-N}, X^{i-N}]$, and define the second order differential operator $V_i$ by 
\begin{equation}\label{Vi}
	V_i := \frac{1}{2} \sum_{j,\, k=1}^d V_i^j V_i^k\partial_{y_jy_k}.
\end{equation}

By analogy to the Stratonovich stochastic system driven by Wiener processes in \cite{Baudoin}, we define the stochastic partial differential operator $S$ as 
\begin{align*}
	S&= \sum_{i=1}^{2N} V_iX^i.
\end{align*}
Since the differential operators $V_i$ are time-homogeneous, we also have $\mathrm{d}S_t  := \,\sum_{i=1}^{2N}V_i\,\mathrm{d}X^i_t$. For $n\ge 1$, let $\int S^n$ denote the $n$-times repeated integral of $S$ and set $\int S^0=\mathrm{Id}$, so that
\begin{align*}
	\int S^n&\, = \int\cdots\int \mathrm{d}S_{t_1}\cdots \mathrm{d}S_{t_n} 
			= \sum_{j_1, j_2, \ldots , j_n=1}^{2N} \!\!\! V_{j_1}\circ \cdots \circ V_{j_n}
				\idotsint \mathrm{d}X_{t_1}^{j_1}\!\ldots\! \mathrm{d}X^{j_n}_{t_n}.		
\end{align*}
The It{\^o}-Taylor series expansion for the flowmap $\varphi_t: Y_0\to Y_t$ corresponding to the stochastic system is given by
\begin{align*}
	\varphi_t = \sum_{n\ge 0}\int S^n \circ \mathrm{Id}.
\end{align*}
We set $\mathbb{S}:=\sum_{n\ge 0}\int S^n=\mathrm{Id} + \int \mathbb{S}\mathrm{d}S$, which describes the action of the flowmap on smooth functions. The central aim of this work is to calculate $\log {\mathbb S}$ in the It\^o framework. 

The computation of the logarithm of this action may be considered as the stochastic analog of a well-known problem in the theory of classical differential equations (motivated, e.g., by numerical considerations, and referred often to as the continuous Baker--Campbell--Hausdorff problem). And it turns out that, provided one uses Stratonovich integrals, the solution for stochastic differential equations is essentially the same as in the classical case. We refer the reader to the next section for more details. 

Displaying the first few terms of $\log {\mathbb S}$ in the It\^o framework may give an idea of the complexity of its expression. We adopt the notation $I_{[i, j]}=[X^i, X^j]$, and  for a repeated It{\^o} integral $I_{j_1, \ldots ,\, j_n} := \idotsint \mathrm{d}X^{j_1}\cdots \mathrm{d} X^{j_n}$. Then the first three terms are as follows
\begin{align*}
	\log {\mathbb S}  = 	&\, \sum_i V_i I_i + \sum_{i,j} V_iV_j \biggl(\frac{1}{2} I_{ij} 
						-\frac{1}{2}\bigl(I_{ji}+I_{[i,j]}\bigr)\biggl) \\
					& \, + \sum_{i,j,k} V_iV_jV_k \biggl(\frac{1}{3} I_{ijk}
						-\frac{1}{6}\bigl(I_{jik}+I_{kij}+I_{[i,j]k}+I_{j[i,k]}\bigl)\\
					&\,\qquad - \frac{1}{6}\bigl(I_{ikj}+I_{jki}+I_{i[j,k]}\bigl)
						+\frac{1}{3}\bigl(I_{kji}+I_{[j,k]i}+I_{k[i,j]}\bigl)\biggl) + \cdots.
\end{align*}

To write the general expression for this expansion, let us introduce some notation. For $f$ a surjection from the set $[n]:=\{1,\ldots,n\}$ to the set $[k]$ (written $f\in {{ Sj}}_{n,k}$), we set 
$$
	d(f):=|\{i<n,f(i)\geq f(i+1)\}|.
$$ 
The set of surjections $f$ from $[n]$ to $[k]$ such that $\forall i\leq k,\ |f^{-1}\{i\}|\leq 2$ is written ${ Sj}^{(2)}_{n,k}$
For a sequence $J=(j_1,\ldots,j_n)$ of elements of $[2N]$ and $\mathcal{A}=A_1\coprod \cdots \coprod A_k=[n]$ an ordered partition of the set $[n]$ into disjoint subsets, we write $I_{\mathcal A}^J$ for the iterated integral $\int\cdots\int \mathrm{d}X_J^{A_1} \cdots \mathrm{d}X_J^{A_k}$, where, for $A_i=\{a_1,\ldots ,a_l\}$, $X_J^{A_i}$ stands for the iterated quadratic covariation $[X^{j_{a_1}},\ldots,X^{j_{a_l}}]:=[X^{j_{a_1}},[X^{j_{a_2}},\ldots,X^{j_{a_l}}]]$. For $f$ as above, we set $\mathcal{A}(f):=f^{-1}(1)\coprod \cdots \coprod f^{-1}(k)$ and 
$$
	S_f:=\sum\limits_{i_1,\ldots,i_n=1}^{2N}V_{i_1}\cdots V_{i_n} I^{\{i_1,\ldots,i_n\}}_{\mathcal{A}(f)}.
$$
 
Our main result reads

\begin{theorem}\label{mainTH}
We have:
$$
	\log ({\mathbb S})=\sum_{n>0} \sum\limits_{n\geq k\geq 1}\sum\limits_{f\in { {Sj}^{(2)}_{n,k} }}
								\frac{(-1)^{d(f)}}{n}\cdot {n-1 \choose{d(f)}}^{-1}S_f.
$$
\end{theorem}

This statement follows from Theorem \ref{mainth}. The restriction of the indexing set to ${ Sj}^{(2)}_{n,k}$ follows from the fact that, for continuous semimartingales, iterated brackets such as $[[X_i,X_j],X_k]$ vanish. When allowing semimartingales with jumps (the article will develop the combinatorial theory of iterated integrals of semimartingales in this more general setting), this restriction disappears.

For example, with $f$ the surjection from $[3]$ to $[2]$, defined by 
$$
	f(1)=2,\ f(2)=1,\ f(3)=2,
$$ 
we obtain $d(f)=1$, and 
$$
	\frac{(-1)^{d(f)}}{n}\cdot {n-1 \choose{d(f)}}^{-1}=\frac{(-1)}{3}\cdot {2 \choose{1}}^{-1}= - \frac{1}{6}
$$ 
and $S_f=\sum\limits_{i,j,k}V_iV_jV_kI_{j[ik]}$, as expected from the low order direct computation given previously.

The rest of the paper develops the formalism necessary to prove the above theorem. Several of the tools that enter our approach are of general interest, and allow to handle the algebraic structures of iterated It\^o integrals. We also show how the same ideas can be applied to the study of linear stochastic matrix differential equations driven by arbitrary semimartingales (Theorem~\ref{thm:nc-Log-Ito}).

\medskip

{\bf{Acknowledgements}}: The first author is supported by a Ram\'on y Cajal research grant from the Spanish government. The third author acknowledges support from the grant ANR-12-BS01-0017, Combinatoire Alg\'ebrique, R\'esurgence, Moules et Applications.


\section{The Strichartz formula}
\label{sect:Strichartz}

Let us recall first the historical as well as technical background of Theorem \ref{mainTH}. Its knowledge will help to enlighten our forthcoming constructions, and make clear to what extend It\^o calculus differs from the Riemann or Stratonovich calculus.

From the seminal 1957 work by K.T.~Chen on the algebraic structures underlying products of iterated integrals \cite{chen0} followed the existence of an exponential solution of the classical nonautonomous initial value problem 
$$
	\dot{Y}(t)=F(t,Y(t)), \quad\ Y(0)=Y_0, 
$$
where $F(t): \mathbb{R}^d \to \mathbb{R}^d$ is a vector field depending continuously on time. 

Over the decades following Chen's work, the explicit formula for this exponential solution was obtained independently by several authors. Mielnik and Plebanski \cite{MP70} as well as Strichartz \cite{Strichartz} calculated the function $\Omega(t)$ such that $Y(t)=\exp(\Omega(t))Y_0$. Using the fact that Stratonovich integrals obey Chen's rules of calculus for iterated integrals, Ben Arous showed soon after the work of Strichartz, that the application domain of this formula extends to stochastic realm \cite{Benarous,Baudoin}.

The explicit expression of the series $\Omega(t)$ is rather intricate. Its most classical formulation involves permutations and iterated integrals of iterated Lie brackets    
$$
	\Omega(t) = \sum_{n>0} \sum_{\sigma \in S_n} \frac{(-1)^{d(\sigma)}}{n^2 {n-1 \choose d(\sigma)}} 
	\int_{\Delta^n_{[0,t]}} [\cdots[F(s_{\sigma(1)}),F(s_{\sigma(2)})], \cdots ],F(s_{\sigma(n)})] {\mathrm{d}}s_1 \cdots  {\mathrm{d}}s_n,
$$       
where the bracket of vector fields follows from their interpretation as differential operators, and the integration domain is the $n$-dimensional simplex 
$$
	\Delta^n_{[0,t]}:=\{(s_1,\ldots,s_n),\  0\le s_1 \le \cdots \le s_n \le t\}.
$$ 
In the formula for $\Omega(t)$, $S_n$ denotes the set of permutations of the set $[n]:=\{1,\ldots, n\}$. The quantity $d(\sigma)$ (already introduced in the more general case of surjections) is called the number of descents of the permutation $\sigma \in S_n$, that is the number of positions in the permutation $(\sigma(1),\ldots, \sigma(n))$, where $\sigma(i)>\sigma(i+1)$, for $i=1,\ldots,n-1$. The formula is known in the literature as Strichartz, Chen--Strichartz or continuous Baker--Campbell--Hausdorff formula. 

It can also be stated by replacing iterated Lie brackets by standard operator products. The formula is then essentially the same except for the scalar coefficient, which becomes $(-1)^{d(\sigma)} n^{-1} {n-1 \choose d(\sigma)}^{-1}$ -the same kind of scalar coefficient as those appearing in Theorem \ref{mainTH}. We refer to \cite{Strichartz} for details on the analytic background, and to \cite{Reutenauer,NCSF} for the underlying combinatorics of free Lie algebras and Lie idempotents.

The derivation of the formula for $\Omega(t)$ relies on a precise understanding of the calculus of iterated integrals, which in turn is based on integration by parts. For two scalar valued indefinite integrals $F(t):= \int^t_0 f(s){\mathrm{d}}s $ and $G(t):=\int^t_0 g(s){\mathrm{d}}s$, recall that
$$
	F(t) \cdot G(t) = \int^t_0 f(s) G(s){\mathrm{d}}s +  \int^t_0 F(s) g(s){\mathrm{d}}s.    
$$
For general iterated integrals 
$$
	F_n(t):=\int_{\Delta^n_{[0,t]}} f_1(s_1)  f_2(s_2) \cdots f_n(s_n) {\mathrm{d}}s_1 \cdots  {\mathrm{d}}s_n,
$$ 
the above product generalizes to Chen's shuffle product formula
\begin{align}
\label{ChenCl}
	F_n(t) \cdot F_m(t) = \sum_{\sigma \in \mathrm{Sh}_{n,m}} 
	\int\limits_{\Delta^{n+m}_{[0,t]}} f_{\sigma^{-1}(1)}(s_1) 
	 \cdots f_{\sigma^{-1}(n+m)}(s_{n+m}){\mathrm{d}}s_1 \cdots  {\mathrm{d}}s_{n+m},
\end{align}
where $\mathrm{Sh}_{n,m}$ is the set of $(n,m)$-shuffles, i.e., permutations $\sigma$ on the set $[n+m]$, such that $\sigma(1)<\cdots < \sigma(n)$ and $\sigma(n+1)<\cdots <\sigma(n+m)$. 

The above product is conveniently abstracted into an algebraically defined shuffle product on words. Let $Y:=\{y_1,y_2,\ldots\}$ be an alphabet, and $Y^*$ the corresponding free monoid of words $\omega=y_{i_1} \cdots y_{i_n}$. The vector space ${\mathbb K}\langle Y\rangle$, which is freely generated by $Y^*$, is a commutative algebra for the shuffle product:
\begin{equation}
\label{shuf-prod}
	v_1\cdots v_p\shu v_{p+1}\cdots v_{p+q}:=
		\sum_{\sigma \in \mathrm{Sh}_{p,q}}v_{\sigma_1^{-1}}\cdots v_{\sigma_{p+q}^{-1}}
\end{equation}
with $v_j\in Y$, $j \in\{1,\ldots ,p+q\}$. We define the empty word $\un$ as unit: $\un \shu v = v \shu \un =v$ for $v \in Y^*$. This product is homogenous with respect to the length of words and can be defined recursively. Indeed, one can show that:
\begin{align}
\label{shuf-prod-rec}
	v_1\cdots v_p \shu v_{p+1}\cdots v_{p+q} = &v_1\big( v_2\cdots v_p \shu v_{p+1}\cdots v_{p+q}  \big) \\
									   &\qquad	+ v_{p+1}\big( v_1\cdots v_p \shu v_{p+2}\cdots v_{p+q} \big). \nonumber 
\end{align}
The shuffle product was axiomatized in the 50's in the seminal works of Eilenberg--MacLane and Sch\"utzenberger \cite{EM,schutz}, and has proven to be essential in many fields of pure and applied mathematics. In \cite{chen1,chen2} Chen studied fundamental groups and loop spaces. In control theory, Chen's abstract shuffle product plays a central role in Fliess' work \cite{fliess}, which combines iterated integrals and formal power series in noncommutative variables into an algebraic approach to nonlinear functional expansions. Reutenauer's monograph on free Lie algebras \cite{Reutenauer} imbeds Chen's work into an abstract Hopf algebra theoretic setting. More recently, Chen's formalism came to prominence in Lyons' seminal theory of rough paths \cite{lyons}. 

When dealing with iterated It\^o integrals of semimartingales, this machinery of shuffle products does not apply any more. One has instead to use the quasi-shuffle product \cite{cemw}. Its definition will be recalled further below. We simply mention for the time being that, although it was discovered and investigated much more recently, it appears to be as important as the shuffle product, both from a theoretical as well as applied point of view. Indeed, it encodes the algebraic structure of discrete sums, as the shuffle product encodes the one of integration maps, and appears in various domains, e.g., multiple zeta values, Rota--Baxter algebras, and Ecalle's mould calculus. The latter is related to the computation of normal forms in the theory of dynamical systems. We refer to \cite{eg,kp2013,fpt,Hoffman} for further historical and technical details on quasi-shuffle algebras.


\section{Semimartingales}
\label{sect:semimart}

Recall that a process $X$ is a {\em semimartingale}, if $X$ has a decomposition $X_t = X_0 + M_t + A_t$ for $t\ge 0$, where $M_0=A_0=0$, and where $M$ is a local martingale and $A$ is an adapted process that is right-continuous with left limits and has finite variation on each finite interval $[0, t]$.  

It is well-known that the space of semimartingales with multiplication forms an associative algebra. The {\em quadratic covariation} or {\em square bracket process} $[X, Y]$ between two semimartingales $X$ and $Y$, is defined via their product as follows
\begin{eqnarray}
\label{eq:covariation}
	X \cdot Y = X_0Y_0 + \int X_-\, \mathrm{d}Y + \int Y_-\, \mathrm{d}X + [X,Y].
\end{eqnarray}
The quadratic covariation of a process $X$ with itself is known as its {\em quadratic variation}. Let $X$, $Y$ and $Z$ be semimartingales, the quadratic covariation satisfies:
\begin{enumerate}

\item
$[X, 0]  = 0$; 

\item
$[X, Y]  = [Y, X]$;

\item 
$[X, [Y, Z]]  =  [[X, Y], Z]$.

\end{enumerate}

Hence, the square bracket process defines a {\em commutative} and {\em associative product} on the space of semimartingales. We refer to the monographs by Protter \cite{Protter} and Jacod \& Shiryaev \cite{Jacod2002} for details.

For notational convenience, we will write from now on iterated integrals of semimartingales as follows:
$$
	\int XY := \int X_-\ \mathrm{d}Y, \quad {\mathrm{and}}\quad  \int X_1 \cdots X_n:=\int (\int X_1 \cdots X_{n-1})_-\, \mathrm{d}X_n.
$$
Iterated brackets are denoted by:
\begin{equation}
\label{star-prod}
	X \star Y:=[X,Y], \quad {\mathrm{and}}\quad  X_1\star  \cdots \star X_n:=[ \cdots [X_1,X_2], \cdots ,X_n].
\end{equation}

From now on we will assume that all processes are normalized so that $X_0=0$. Terms such as $X_0Y_0$, can therefore be ignored in products of stochastic integrals, e.g., as in equation (\ref{eq:covariation}).

Let us briefly illustrate the combinatorial nature of iterated It\^o integrals, before turning to the general, and more abstract picture. Recall first that, for arbitrary semimartingales $A,B,C,D$ and $X:=\int A_-\ {\mathrm{d}}B$, $Y:=\int C_-\ {\mathrm{d}}D$, we have
\begin{equation}
\label{qshheq}
	X \cdot Y = \int (XC)_-\ \mathrm{d}D + \int (AY)_-\ \mathrm{d}B + \int (AC)_-\ \mathrm{d}[B,D],
\end{equation}
so that, for example, for $B:=\int {\mathrm{d}}B$, $Y:=\int C_-\ {\mathrm{d}}D$
\allowdisplaybreaks{
\begin{eqnarray*}
	B \cdot Y 	&=& \int (CB)_-\ \mathrm{d}D + \int Y_-\ \mathrm{d}B + \int C_-\ \mathrm{d}[B,D]\\
		       	&=& \int CDB + \int (\int C_-\ \mathrm{d}B + \int B_-\ \mathrm{d}C +[B,C])_-\ {\mathrm{d}}D 
			+ \int C_-\ \mathrm{d}[B,D]\\
			&=& \int BCD + \int CBD + \int CDB + \int [B,C]D + \int C[B,D].
\end{eqnarray*}}
Whereas the first three terms of the expansion, i.e., $\int BCD+\int CBD+\int CDB$, are those that would also appear in the shuffle product expansion, the last two terms arise from the bracket operation. These extra terms are typical outcomes of what distinguishes quasi-shuffle and shuffle computations. The algebra underlying the quasi-shuffle calculus has been explored in \cite{fpt,npt}. In a nutshell, the bracket terms that appear in the above product require the replacement of permutations in the calculation of the Chen--Strichartz formula by the larger class of surjections.


\section{Quasi-shuffle algebra}
\label{sect:quasiShuf}

We recall now the definition and basic properties of quasi-shuffle algebras. In spite of the fact that such algebras encode naturally It\^o integral calculus, they appeared only in a few papers in the context of stochastic integration, see e.g.~the works by Gaines, and Liu and Li \cite{Gaines,LiLiu} and \cite{cemw}. As the name indicates, quasi-shuffle algebras are obtained as a deformation of classical shuffle algebras. It is generally agreed that the idea can be traced back to the work of P.~Cartier on free commutative Rota--Baxter algebras \cite{Cartier}. However, it was formalized only recently by M.~Hoffman \cite{Hoffman}. The link between Hoffman's ideas and commutative Rota--Baxter algebras was explored in \cite{eg}.

Abstractly, a quasi-shuffle algebra is defined as a commutative algebra $B$ with product $\bullet$, which is equipped with two extra bilinear products denoted $\uparrow$ and $\downarrow$ (known as ``half-shuffles''), such that for $x, y, z \in B$ one has
\allowdisplaybreaks{
\begin{eqnarray}
\label{asheax}
	 x \uparrow y &=& y \downarrow x,\\
	(x\bullet y)\uparrow z &=& x\bullet (y\uparrow z) \\
	(x \uparrow y) \uparrow z &=& x \uparrow (y \uparrow z + y \downarrow z + y \bullet z).
\end{eqnarray}}

One writes usually 
$$
	x \ash y:= x \uparrow y + x \downarrow y + x\bullet y,
$$ 
and calls $\ash$ the quasi-shuffle product. In particular the last axiom then reads $(x\uparrow y)\uparrow z=x\uparrow (y\ash z)$. 
Note that when the product $\bullet$ on $B$ is the null product, then one recovers the usual axioms for shuffle algebras \cite{EM,schutz}. 

Let us mention that the ``deformation'' induced by the $\bullet$ product can be understood in terms of -- weight-one -- commutative Rota--Baxter algebras. The term $y \uparrow z + y \downarrow z + y \bullet z$ can then be interpreted as the so-called double Rota--Baxter product. We refrain however from developing these ideas here, since they are only indirectly relevant to our present purposes. The interested reader is referred to the survey paper \cite{kp2013} for an overview of the links between the theories of Rota--Baxter algebras, integral calculus and (quasi-)shuffle algebras. 

The standard example of a quasi-shuffle algebra, studied in detail in \cite{Hoffman}, is provided by the tensor algebra $T(A):=\bigoplus_{n\in \mathbb{N}}A^{\otimes n}$ over a commutative algebra $(A,\ast)$. The three products $\uparrow$, $\downarrow$, $\bullet$ are defined inductively by: $a \bullet b:=a \ast b$, $a \uparrow b := ba$, $a \downarrow b :=ab$, and
\allowdisplaybreaks{
\begin{eqnarray*}
	a_1 \cdots a_n \uparrow b_1 \cdots b_m
		&:=& (a_1 \cdots a_{n-1} \ash b_1 \cdots b_m)a_n,\\
	a_1 \cdots a_n \downarrow b_1 \cdots b_m 
		&:=& (a_1 \cdots a_n \ash b_1 \cdots b_{m-1})b_m,\\
	a_1 \cdots a_n \bullet b_1 \cdots b_m
	&:=& (a_1 \cdots a_{n-1} \ash b_1 \cdots b_{m-1})(a_n\ast b_m),
\end{eqnarray*}}
where we used the common word notation $a_1 \cdots a_n$ for $a_1 \otimes  \cdots \otimes a_n \in A^{\otimes n}$.

For example, the product of two words of length two gives explicitly
\allowdisplaybreaks{
\begin{align*}
		\lefteqn{a_1a_2 \ash b_1b_2 = (a_1 \ash b_1b_2)a_2 + (a_1 a_2 \ash b_1) b_2 
										+ (a_1 \ash b_1)(a_2 \ast b_2)}\\
		&= b_1b_2a_1a_2 + a_1b_1b_2a_2 + b_1a_1b_2a_2 + a_1a_2b_1b_2 + b_1a_1a_2b_2 + a_1b_1a_2b_2 \\ 
		& \quad\ +  b_1(a_1\ast b_2)a_2 + (a_1\ast b_1)b_2a_2 + a_1(a_2\ast b_1)b_2 + (a_1\ast b_1)a_2b_2 \\
		& \quad\  + a_1b_1(a_2\ast b_2) +  b_1a_1(a_2\ast b_2) +  (a_1 \ast b_1)(a_2 \ast b_2). 
\end{align*}}

\medskip

Recall from Section \ref{sect:Strichartz} that the recursive description of shuffle product is complemented by its definition in terms of permutations. Similarly, the above recursive definition of the quasi-shuffle product has an explicit presentation in terms of surjections. Concretely, let $f$ be a surjective map from $[n]=\{1, \ldots ,n\}$ to $[p]$. We set:
$$
	f(a_1 \cdots a_n) := (\prod\limits_{j \in f^{-1}(1)}{^{\!\!\!\!\!\!\!\!\!^{\mathlarger{\ast}}}}\,\, a_j) \otimes  
					\cdots \otimes (\prod\limits_{j\in f^{-1}(p)}{^{\!\!\!\!\!\!\!\!\!^{\mathlarger{\ast}}}}\,\,  a_j)\in A^{\otimes p},
$$
so that for $f$ from, say, $[4]$ to $[2]$ given by $f(1)=1$, $f(2)=2$, $f(3)=1$, $f(4)=2$, we find $f(a_1\cdots a_4) = (a_1 \ast a_3) \otimes (a_2\ast a_4)$. 

Then, we obtain:
\begin{equation}
\label{def:qShSurjections}
	a_1\cdots a_n \ash b_1\cdots b_m := \sum_f f(a_1\cdots a_nb_1\cdots b_m),
\end{equation}
where $f$ runs over all surjections from the set $[n+m]$ to the set $[k]$, for $\max(n,m) \le k \le m+n$, and such that $f(1) < \cdots  < f(n)$, $f(n+1) < \cdots < f(n+m)$.

Let us write now ${\mathcal T}$ for the tensor algebra over the algebra $\mathcal S$ of semimartingales, equipped with the $\star$ product defined in \eqref{star-prod}, so that from now on $X_1 \cdots X_n$ denotes a tensor product of semimartingales in ${\mathcal S}^{\otimes n} \subset \mathcal T$, and $\int X_1 \cdots X_n$ the corresponding iterated stochastic integral. We finally obtain the analog for iterated integrals of semimartingales of the usual Chen formulas of iterated integrals. Recall that the latter hold for either Stratonovich or indefinite Riemann integrals.

\begin{proposition}
\label{prodstoc}
The product of two iterated stochastic integrals is given by:
\allowdisplaybreaks{
\begin{eqnarray*}
	\int X_1 \cdots X_n \cdot \int Y_1 \cdots Y_m	
			&=& \int (X_1 \cdots X_n\ash Y_1 \cdots Y_m)\\
			&=& \sum_f \int f(X_1 \cdots X_nY_1 \cdots Y_m),
\end{eqnarray*}}
where, as above, $f$ runs over all surjections from the set $[n+m]$ to the set $[k]$, for $\max(n,m) \le k \le m+n$, and such that $f(1) < \cdots <f(n)$, $f(n+1)< \cdots <f(n+m)$.
\end{proposition}
For example, the product of two iterated stochastic integrals gives
$$
	\int X_1 X_2 \cdot \int X_3=\int \big(X_1X_2X_3 + X_1X_3X_2 + X_3X_1X_2 + (X_1\star X_3)X_2+X_1(X_2 \star X_3)\big).
$$

The proposition follows from the observation that the inductive rules for the quasi-shuffle product in the tensor algebra give the pattern obeyed by products of iterated integrals of semimartingales. Namely, setting for $X:= \int A_-{\mathrm{d}}B$ and $Y:= \int C_-{\mathrm{d}}D$, we have
\allowdisplaybreaks{
\begin{eqnarray*}
	X \uparrow Y	&:=& \int\big(A(\int C_-dD)\big)_- {\mathrm{d}}B,\\  
	X \downarrow Y&:=& \int \big((\int A_-dB)C\big)_- {\mathrm{d}}D, \\
	X\bullet Y		&:=& \int (AC)_- {\mathrm{d}}[B,D].
\end{eqnarray*}}


\section{Surjections}
\label{sect:surjections}

This section presents a concise and mostly self-contained account on the modern algebraic theory of surjections, originating independently from F.~Hivert's Ph.D.-thesis \cite{FH} and from the work by Chapoton on the permutohedron \cite{Chapoton}. We refer to the works  \cite{fpt,npt} for more details on the subject.

Let us write $\Surj_{n,p}$ for the set of surjective maps from the set $[n]:=\{1,\ldots,n\}$ to the set $[p]$, $\SURJ_{n,p}$ for its linear span, $\Surj_n$ for the union of the $\Surj_{n,p}, \ p\leq n$, and 
$$
	\SURJ_n := \bigoplus\limits_{1\leq p\leq n} \SURJ_{n,p}.
$$
The linear span of all surjections is denoted
$$
	\SURJ := \bigoplus\limits_{n,p}\SURJ_{n,p},
$$ 
and will be called from now on the set of {\it{surjective functions}}. 

Let us mention for completeness sake that surjective functions in this sense are often referred to as \it word quasisymmetric functions \rm in the literature on algebraic combinatoric. This is because they can be encoded by formal sums of words over an ordered alphabet. The set $\SURJ$ is then written $\mathbf{WQSym}$ (this is the notation used in the articles we quoted for further details on the underying theory). This interpretation, which we will not use in this work, permits to deduce automatically certain properties for $\SURJ$ from general properties of words. See the references \cite{npt,fpt} for more details.

The vector space $\SURJ$ is naturally equipped with a Hopf algebra structure, through its action on quasi-shuffle algebras \cite{npt}. However, we will make use here only of the algebra structure. As we just saw, quasi-shuffle algebras are closely related to stochastic integration, and this is the reason why $\SURJ$ will prove to provide an appropriate algebraic framework for what is going to be presented in the next sections.

Let us consider a word $w$ over the integers (a sequence of integers) whose set of letters is $I=\{i_1, \ldots ,i_n\}$ (e.g. $w=35731$, $I=\{1,3,5,7\}$). Let us write $f$ for the unique increasing bijection that sends $I=\{i_1, \ldots,i_n\}$ to $\{1,\ldots,n\}$ (e.g. $f(1)=1,f(3)=2,f(5)=3,f(7)=4$). The \it packing map \rm $\pack$ is the induced map on words (for example $\pack(35731):=f(3)f(5)f(7)f(3)f(1)=23421$).

\begin{definition} 
The product in $\SURJ$ of $f \in \Surj_{n,k}$ with $g \in \Surj_{m,l}$ is defined by $f \diamond g:= \sum_h h,$ where $h$ runs over the elements in $\Surj_{n+m,i+j}$, $\max(l,k) \leq i+j \leq l+k$ such that 
$$
	\pack(h(1) \cdots h(n)) = f(1) \cdots f(n),\quad \pack(h(n+1) \cdots h(n+m))=g(1) \cdots g(m).
$$
\end{definition}

This product is associative and unital (the unit can be understood as the unique surjection from the emptyset $\emptyset =:[0]$ to itself), but it is not commutative. 

\begin{proposition}\label{wqsym}
The linear span of all surjection, $\SURJ$, equipped with the $\diamond$ product, is an associative, unital, non-commutative algebra.
\end{proposition}

Associativity follows by noticing that the product $f \diamond g \diamond j$ of three surjections, where $j \in \Surj_{p,q}$ is obtained as the sum of all surjections $h$ with $pack(h(1) \cdots h(n))=f(1) \cdots f(n)$, $pack(h(n+1) \cdots h(n+m))=g(1) \cdots g(m)$, $pack(h(n+m+1) \cdots h(n+m+p))=j(1) \cdots j(p)$.


\section{Descents and a It\^o-type BCH formula}
\label{IBCH}

It turns out that, similar to the classical theory of iterated integrals, the most interesting computations that will take place later in this article on iterated stochastic integrals do not involve the full algebra $\SURJ$, but only a small subalgebra, known as the {\it{descent algebra}} or {\it{algebra of noncommutative symmetric functions}} $\NCSF$. For more details the reader is refereed to the standard references \cite{NCSF,Reutenauer}.

As an algebra, $\NCSF$ is the free graded associative unital algebra over generators $1_n$, indexed by non-negative integers. We denote the product $\ast$ and set $1_{n,m}:=1_n \ast 1_m$. In general, for a sequence $\overline n:=n_1, \ldots ,n_k$ of integers, we write $1_{\overline n}:=1_{n_1} \ast \cdots \ast 1_{n_k}$. As a vector space, $\NCSF$ is simply the linear span of the $1_{\overline n}$.

A surjection $f$ in $\Surj_{n}$ is said to have a descent in position $i<n$ if and only if $f(i) \geq f(i+1)$. The set of all descents of $f$ is written $\mathrm{Desc}(f)$ and
$$
	\mathrm{Desc}(f) := \{i<n,\ f(i) \geq f(i+1)\}.
$$
The number $d(f)$ that appears in Theorem \ref{mainTH} is the number of descents of $f$.

We also set, for $I \subset [n-1]$, 
\allowdisplaybreaks{
\begin{eqnarray*}
	\mathrm{Desc}_I^n	&:=& \{f\in \Surj_n,\ \mathrm{Desc}(f)=I\},\\
	D_I^n	&:=& \sum\limits_{f\in \mathrm{Desc}_I}f\in \Surj ,
\end{eqnarray*}}
and
\allowdisplaybreaks{
\begin{eqnarray*}
	\mathrm{Desc}_{\subseteq I}^n	&:=& \{f\in \Surj_n,\ \mathrm{Desc}(f)\subseteq I\},\\
	D_{\subseteq I}^n		&:=& \sum\limits_{f\in \mathrm{Desc}_{\subseteq I}}f\in \Surj .
\end{eqnarray*}}
When the value of $n$ is obvious from the context, we will abbreviate $D_I^n$ by $D_I$, and similarly for other symbols. Notice that $\mathrm{Desc}_{\subseteq I}=\sum_{J\subseteq I}\mathrm{Desc}_J$, so that, by M\"obius inversion in the poset of subsets of the set $[n-1]$, 
\begin{equation}
\label{moebius}
	\mathrm{Desc}_{I}=\sum\limits_{J\subseteq I}(-1)^{|I|-|J|}\mathrm{Desc}_{\subseteq I}.
\end{equation}
The subsets $\mathrm{Desc}_I$ form a decomposition of $\Surj_n$ into a family of disjoint subsets, from which it follows that $D_I$ and (by a triangularity argument) $D_{\subseteq I}$ form two linearly independent families in $\SURJ_n$.

\begin{lemma}
The map $\iota$ from $\NCSF$ to $\SURJ$ defined by:
$$
	\iota(1_{\overline n}):=D_{\subseteq \{n_1,n_1+n_2,\ldots , n_1+ \cdots +n_{k-1}\}}
$$
is an injective algebra map from $\NCSF$ into $\SURJ$.
\end{lemma}

\begin{proof}
Injectivity follows immediately from the linear independency of the $D_{\subseteq I}$. Let us show that $\iota$ is an algebra map. For arbitrary $\overline n = n_1, \ldots ,n_k$, $\overline m = m_1, \ldots ,m_l$, we have
$$
	\iota(1_{\overline n}\ast 1_{\overline m})
	=D_{\subseteq \{n_1,n_1+n_2,\ldots ,n_1+ \cdots +n_k,n_1+ \cdots + n_k+m_1, 
	\ldots ,n_1+ \cdots + n_k + m_1 +  \cdots +m_{l-1}\}}.
$$
On the other hand, $\iota(1_{\overline n}) \diamond \iota(1_{\overline m})$ is, by definition of the $\diamond$ product in $\SURJ$, the sum of all surjections $f \in \Surj_{n_1+ \cdots +n_k+m_1+ \cdots +m_l}$ such that $\pack(f(1) \cdots f(n_1+ \cdots +n_k))$ lies in 
$$
			\mathrm{Desc}_{\subseteq \{n_1,n_1+n_2, \ldots ,n_1+ \cdots +n_{k-1}\}}
$$ 
and $\pack(f(n_1+ \cdots +n_k+1) \cdots f(n_1+ \cdots +n_k+m_1+ \cdots +m_l))$ lies in
$$
	\mathrm{Desc}_{\subseteq \{m_1,m_1+m_2, \ldots , m_1+ \cdots +m_{l-1}\}}.
$$ 
Since there is no constraint on the relative values of $f(n_1+ \cdots +n_k)$ and $f(n_1+ \cdots +n_k+1)$, the statement of the lemma follows.
\end{proof}

The following theorem is the equivalent, in the quasi-shuffle framework, of the classical continuous Baker--Campbell--Hausdorff theorem, which computes, among others, the logarithm of the solution of a -- matrix-valued -- linear differential equation. Theorem \ref{mainth} will appear to play the same role for matrix stochastic linear differential equations. It was first stated in \cite{npt}, but the proof given in that article is indirect and relies on structure arguments from the theory of noncommutative symmetric functions. Stating those results, which are scattered in the literature on algebraic combinatorics, would go beyond the scope of this work. We propose therefore a simple and self-contained proof, which is reminiscent of the solution to the classical Baker--Campbell--Hausdorff problem stated in \cite{Reutenauer}.

Let us set $I:=\sum_{n=0}^\infty \iota(1_n)=:\sum_{n=0}^\infty p_n$, where we write $p_n$ for the identity map of the set $[n]$ viewed as an element of $\Surj_n$.

\begin{theorem}\label{mainth}
We have, in $\SURJ$,
\allowdisplaybreaks{
\begin{eqnarray*}
	\log(I) &=& \sum\limits_{n=1}^\infty \sum\limits_{I\subseteq [n-1]}\frac{(-1)^{|I|}}{|I|+1}\cdot D_{\subseteq I}^n \\
		 &=& \sum\limits_{n=1}^\infty \sum\limits_{I\subseteq [n-1]}\frac{(-1)^{|I|}}{n}\cdot {n-1\choose |I|}^{-1} D_{ I}^n.
\end{eqnarray*}}
\end{theorem}

\begin{proof}
The first part of the statement follows from the computation of the logarithm of $\sum_{n=0}^\infty 1_n$ in $\Sym$, and from the previous Lemma. Indeed
$$
	\log\big(\sum\limits_{n=0}^\infty 1_n\big)
		=\sum\limits_{i=1}^\infty \frac{(-1)^{n-1}}{n}\big(\sum\limits_{i=1}^\infty 1_i\big)^n
		=\sum\limits_{{\overline n}=i_1,\ldots,i_k}\frac{(-1)^{k-1}}{k}\cdot 1_{\overline n}.
$$

Let us now expand the second term of the identity in the theorem in terms of the $D_{ I}^n$. This yields the coefficient of $D_I^n$, $I=\{i_1, \ldots,i_k\}$ given by:
\begin{eqnarray*}
	\sum\limits_{I\subseteq J\subset [n-1]}\frac{(-1)^{|J|}}{|J|+1}&=&\sum\limits_{j=0}^{n-k-1}\frac{(-1)^{j+k}}{j+k+1}{n-k-1\choose j}\\
	&=&(-1)^k\sum\limits_{j=0}^{n-k-1}\int\limits_0^1(-1)^j{n-k-1\choose j}x^{k+j}dx\\
	&=&(-1)^{k}\int\limits_0^1(1-x)^{n-k-1}x^kdx=(-1)^k\frac{1}{n}{n-1\choose k}^{-1},
\end{eqnarray*}
from which the theorem follows.
\end{proof}

In view of Proposition \ref{prodstoc}, Theorem \ref{mainTH} also follows.

The elements  $\sum_{I\subseteq [n-1]}\frac{(-1)^{|I|}}{n}\cdot {n-1\choose |S|}^{-1} D_{ I}^n$ are  analogs in $\SURJ$ of the celebrated Solomon idempotents, see \cite{npt} for further details.


\section{Noncommutative stochastic calculus}
\label{sect:ncStochCalc}

In the previous sections of the article, we investigated quasi-shuffle-type properties of iterated integrals of semimartingales. We also showed the strong relationship between them and properties of surjections. Now Proposition \ref{prodstoc} is restated and generalized:

\begin{proposition}\label{prodrule}
The product of $k$ iterated stochastic integrals of semimartingales is given by:
\begin{align*}
	&\Big(\int X_1^1 \cdots X_{n_1}^1\Big)\cdot\ \cdots\ \cdot \Big(\int X_1^k  \cdots X_{n_k}^k\Big)\\
	&\qquad = \sum_{f\in {\mathrm{Desc}}_{\subseteq\{n_1, \ldots ,n_1+ \cdots +n_{k-1}\}}} 
	\int f\Big(X_1^1 \cdots X_{n_1}^1 \cdots X_1^k \cdots X_{n_k}^k\Big).
\end{align*}
More generally, for $f_1 \in \Surj_{n_1}, \ldots ,f_k \in \Surj_{n_k}$:
\begin{align*}
	&f_1\Big(\int X_1^1 \cdots X_{n_1}^1\Big)\cdot\ \dots\ \cdot f_k\Big(\int X_1^k \cdots X_{n_k}^k\Big)\\
	&\qquad = (f_1\diamond  \cdots \diamond f_k)\Big(\int X_1^1 \cdots X_{n_1}^1 \cdots X_1^k \cdots X_{n_k}^k\Big).
\end{align*}
\end{proposition}

In the last formula, from which the first follows, it is implicitly assumed that the action of $\Surj_k$ on $\int X_1 \cdots X_k$ is extended linearly to the linear span of $\Surj_k$, that is, the action of a linear combination of surjections $f_i$ is the linear combination of the actions of the $f_i$. 

We let the reader check that the $k=2$ case of the last formula follows from the definition of the $f\Big(\int X_1^1 \cdots X_{n}^1\Big)$ and from Proposition \ref{prodstoc}. The general case follows by induction.

For example, the expansion of the triple product $(\int X) \cdot (\int Y_1Y_2) \cdot (\int Z_1Z_2)$ includes terms such as:
$$
	\int XY_1(Y_2 \star Z_1)Z_2,\ \int (X \star Y_1)(Y_2 \star Z_1)Z_2\ {\rm{and}}\ \int Z_1(X\star Y_1\star Z_2)Y_2.
$$	

The purpose of the present section is to extend this picture to the operator setting, that is, to iterated stochastic integrals of, say, $n \times n$ square matrices $M=M(X^{i,j})_{1\leq i,j\leq n}$ whose entries $M^{i,j}=X^{i,j}$ are semimartingales. Note that we write the indices of the entries as exponents for notational convenience in forthcoming computations. The set of such matrices is denoted $\mathfrak M$.

For matrices $M_1, \ldots ,M_k \in \mathfrak M$, we set:
$$
	\int M_1 \cdots M_k := \Big(\sum\limits_{i_1,\ldots,i_k} \int M_1^{i,i_1}M_2^{i_1,i_2}
	\cdots M_k^{i_k,j}\Big)^{i,j}_{1\leq i,j\leq n},
$$
and for $f \in \Surj_{k,l}$,
$$
	f\Big(\int M_1\cdots M_k\Big):=\bigg(\sum\limits_{i_1,\ldots ,i_k}\int 
	f\Big(M_1^{i,i_1}M_2^{i_1,i_2}\cdots M_k^{i_k,j}\Big)\bigg)^{i,j}_{1\leq i,j\leq n}.
$$
Last, if $F=\sum_i \lambda_i f_i$ is a linear combination of surjections in $\SURJ_k$, then we set:
$$
	F\Big(\int M_1\cdots M_k\Big) := \sum_i \lambda_i f_i\Big(\int M_1\cdots M_k\Big).
$$

The product rule of Proposition \ref{prodrule} applies (entry-wise), and we obtain, for $M_1, \ldots ,M_k,N_1, \ldots ,N_m \in \mathfrak M$:
\begin{align*}
	&\int M_1 \cdots M_{k} \cdot \int N_1 \cdots N_m\\
	&\quad = \Big(\sum_{f \in {\mathrm{Desc}}_{\subseteq\{k\}}} \sum \limits_{i_1, \ldots ,i_{k+m}}
		\int f\big(M_1^{i,i_1} \cdots M_{k}^{i_k,i_{k+1}}N_1^{i_{k+1},i_{k+2}} 
		\cdots N_m^{i_{k+m},i_j}\big)\Big)^{i,j}_{1\leq i,j\leq n}\\
 	&\quad = \sum_{f \in {\mathrm{Desc}}_{\subseteq\{k\}}}f\Big(\int M_1 \cdots M_{k} N_1 \cdots N_m\Big).
\end{align*}
 
For example, let $k=2,m=1$, and consider $2 \times 2$ matrices, this gives for the first entry of the product:
\begin{align*}
	\Big(\int M_1M_2 \cdot \int N_1\Big)^{1,1}
	 &= \sum_{i,j\leq 2}\int \Big(M_1^{1,i}M_2^{i,j}N_1^{j,1}+M_1^{1,i}N_1^{j,1}M_2^{i,j}+N_1^{j,1}M_1^{1,i}M_2^{i,j}\\
	&\quad\quad + M_1^{1,i}(M_2^{i,j} \star N_1^{j,1})+(M_1^{1,i} \star N_1^{j,1})M_2^{i,j}\Big).
\end{align*}	

For higher products we obtain similarly:

\begin{proposition}
For $M_1^1, \ldots, M_{n_1}^1,  M_1^k,  \ldots , M_{n_k}^k \in \mathfrak M$, we have:
\begin{align*}
	&\Big(\int M_1^1 \cdots M_{n_1}^1\Big)\cdot\  \cdots\ \cdot \Big(\int M_1^k \cdots M_{n_k}^k\Big)\\
	&\quad = \sum_{f \in {\mathrm{Desc}}_{\subseteq\{n_1, \ldots ,n_1+ \cdots +n_{k-1}\}}} 
	f\Big(\int M_1^1 \cdots M_{n_1}^1\cdots M_1^k\cdots M_{n_k}^k\Big)
\end{align*}
and more generally, for $f_1 \in \Surj_{n_1}, \ldots ,f_k \in \Surj_{n_k}$ we obtain:
\begin{align*}
	&f_1\Big(\int M_1^1\cdots M_{n_1}^1\Big)\cdot\ \dots\ \cdot f_k\Big(\int M_1^k\cdots M_{n_k}^k\Big)\\
	&\quad = (f_1\diamond \cdots \diamond f_k)\Big(\int M_1^1\cdots M_{n_1}^1\cdots M_1^k\cdots M_{n_k}^k\Big).
\end{align*}
\end{proposition}

The proposition follows from the linear case (Proposition \ref{prodrule}) by expanding entry-wise the products of matrices.

We are now in the position to calculate the logarithm of the It\^o-Taylor series.

\begin{theorem}
\label{thm:nc-Log-Ito}
For an arbitrary matrix $M \in \mathfrak M$ and $X =\sum_{n=0}^\infty \int M^n$ the (formal) solution of the stochastic differential equation ${\mathrm{d}}X = X_-{\mathrm{d}}M$, $X_0:=Id$, we have:
$$
	\log(X)=\sum\limits_{n=1}^\infty \sum\limits_{I \subseteq [n-1]}\frac{(-1)^{|I|}}{n}\cdot {n-1\choose |I|}^{-1} 
	D_{ I}^n\int M^n.
$$
\end{theorem}

This formula may provide the basis for interesting numerical properties. For instance, truncating the expansion of $\log(X)$ at order $k$, that is, looking at 
$$
	\sum_{n=1}^k \sum_{I\subseteq [n-1]}\frac{(-1)^{|I|}}{n}\cdot {n-1\choose |S|}^{-1} D_{ I}^n\int M^n,
$$ 
and applying the exponential map, can be expected (in view of similar phenomena in the deterministic case) to provide a better approximation to $X$, than the truncation of the original expansion $\sum_{i=0}^k \int M^n$.

The first few terms of the expansion of $\log(X)$ read:
\begin{align*}
	& \log(X) =\int M+\left( \frac{1}{2}(12)-\frac{1}{2}((21)+(11))\right) \int M^2\\
			& \quad +\Big(\frac{1}{3}(123)-\frac{1}{6}\big((213)+(312)+(112)+(212)\big)\\
			& \quad -\frac{1}{6}\big((132)+(231)+(122)+(121))\\
			& \quad +\frac{1}{3}((321)+(211)+(111)+(221)\big)\Big)\int M^3 + \cdots.
\end{align*}
We remind the reader that we represent a surjection $f \in \Surj_k$ by the sequence of its values $(f(1) \cdots f(k))$.

For example, for $2\times 2$ matrices, the $(1,2)$-entry of the term $(11) \int M^2$ is given by $\sum_{i\leq 2}\int M^{1,i}\star M^{i,2}$; the one of $(231)\int M^3$ reads $\sum_{i,j\leq 2}\int M^{j,2}M^{1,i}M^{i,j}$; the one of $(221)\int M^3$ reads $\sum_{i,j\leq 2}\int M^{j,2}(M^{1,i}\star M^{i,j}).$



\footnotesize
\begin{thebibliography}{99}


\bibitem{Baudoin}
	F.~Baudoin,
	An Introduction To The Geometry Of Stochastic Flows,
	World Scientific Publishing Company, Singapour, 2005.
	
\bibitem{Benarous} 
	G\'erard Ben Arous, 
	{\emph{Flots et series de Taylor stochastiques}},
	Probab.~Th.~Rel.~Fields {\bf{81}}, no.~1, 29 (1989). 

\bibitem{Cartier} 
	P.~Cartier, 
	{\emph{On the structure of free Baxter algebras}},
         Adv.~Mathematics {\bf{9}}, 253 (1972). 
         
\bibitem{Chapoton} 
	F.~Chapoton, 
	{\emph{Alg\`ebres de Hopf des permutah\`edres, associah\`edres et hypercubes}},
	Adv.~in Mathematics {\bf{150}}, 264 (2000). 


\bibitem{cemw} 
	C.~Curry, K.~Ebrahimi-Fard, S.J.~Malham, A.~Wiese,
	{\emph{L\'evy Processes and Quasi-Shuffle Algebras}},
	 Stochastics {\bf{86}}, 632 (2014). 

\bibitem{chen0} 
	K.T.~Chen, 
	{\emph{Integration of paths, geometric invariants and a generalized Baker--Hausdorff formula}}, 
	Ann.~Math.~{\bf{65}}, 163 (1957). 

\bibitem{chen1} 
	K.T.~Chen, 
	{\emph{Algebras of iterated path integrals and fundamental groups}}, 
	Transactions of the American Mathematical Society {\bf{156}}, 359  (1971). 
	
\bibitem{chen2} 
	K.T.~Chen, 
	{\emph{Iterated path integrals}}, 
	Bulletin of the American Mathematical Society {\bf{83}}(5), 831 (1977). 
	

\bibitem{EM} 
	S.~Eilenberg, S.~Mac Lane,
	{\emph{On the Groups $H(\pi, n)$}}, 
	Annals of Mathematics, Second Series, {\bf{58}}, no.~1, 55 (1953). 
	
\bibitem{ku2002}	
	K.~Ebrahimi--Fard, 
	{\emph{Loday-type algebras and the Rota--Baxter relation}} 
	Letters in Mathematical Physics {\bf{61}}(2), 139 (2002). 

\bibitem{eg} 
	K.~Ebrahimi--Fard, L. Guo,
	{\emph{Quasi-shuffles, Mixable Shuffles, and Hopf Algebras}},
	Journal of Algebraic Combinatorics \textbf{24}(1), 83,  (2006).


\bibitem{kp2013} 
	K.~Ebrahimi-Fard, F.~Patras, 
	{\emph{La structure combinatoire du calcul int\'egral}}, 
	Gazette des Math\'ematiciens {\bf{138}}, (2013).

\bibitem{kf2013} 
	K.~Ebrahimi-Fard, F.~Patras, 
	{\emph{The pre-Lie structure of the time-ordered exponential}} 
	Letters in Mathematical Physics {\bf{104}}, no.~10, 1281 (2014).

\bibitem{fliess} 
	M.~Fliess, 
	{\emph{Fonctionnelles causales non lin\'eaires et ind\'etermin\'ees non commutatives}} 
	Bulletin de la soci\'et\'e math\'ematique de France {\bf{109}}, 3 (1981).

\bibitem{fpt} 
	L.~Foissy, F.~Patras, J.-Y.~Thibon, 
	{\emph{Deformations of shuffles and quasi-shuffles}},
	preprint, arXiv:1311.1464.

\bibitem{Gaines} 
	J.~Gaines,
	{\emph{The algebra of iterated stochastic integrals}},
	Stochastics and Stoch.~Reports \textbf{49}, 169 (1994). 
 
\bibitem{NCSF}
	I.~M.~Gelfand, D.~Krob, A.~Lascoux, B.~Leclerc, V.~S.~Retakh, J.-Y.~Thibon,
	{\emph{Noncommutative symmetric functions}},	
	Adv.~in Mathematics {\bf 112}, 218 (1995). 

\bibitem{FH}
	F.~Hivert,
	{\emph{Combinatoire des fonctions quasi-sym\'etriques}},
	Th\`ese de doctorat (Doctoral thesis), 1999.

\bibitem{Hoffman} 
	M.~E.~Hoffman, 
	{\emph{Quasi-shuffle products}} 
	Journal of Algebraic Combinatorics \textbf{11}, 49 (2000).


\bibitem{Jacod2002} 
	J.~Jacod, A.N.~Shiryaev,  
	Limit Theorems for Stochastic Processes.
	Berlin, etc: Springer 2002.



\bibitem{LiLiu} 
	C.W~Li, X.Q.~Liu, 
	{\emph{Algebraic structure of multiple stochastic integrals with respect to Brownian motions and Poisson processes}}, 
	Stochastics and Stoch.~Reports \textbf{61}, 107 (1997). 

\bibitem{lyons} 
	T.~Lyons, M.~Caruana,  T.~L\'evy, 
	{\emph{Differential Equations Driven by Rough Paths}}, 
	\'Ecole d'\'et\'e des probabilit\'es de Saint-Flour XXXIV?2004 (J. ~Picard, ed.). 
	Lecture Notes in Mathematics {\bf{1908}}.



\bibitem{MP70}
    	B.~Mielnik, J.~Pleba\'nski,
    	{\textsl{Combinatorial approach to Baker--Campbell--Hausdorff exponents}}
    	Ann.~Inst.~Henri Poincar\'e A {\bf{XII}}, 215 (1970). 

\bibitem{npt} 
	J.C.~Novelli, F.~Patras, J.-Y.Thibon,
	{\emph{Natural endomorphisms of quasi-shuffle Hopf algebras}} 
	2011, arXiv:1101.0725v1. To appear in Bull. Soc. Math. de France.

 

\bibitem{Fred94} 
	F.~Patras, 
	{\emph{L'alg\`ebre des descentes d'une big\`ebre gradu\'ee}}, 
	J.~Algebra {\bf{170}}, 2, 547 (1994). 

\bibitem{Protter} 
	P.~Protter, 
	{Stochastic Integration and Differential Equations}, 
	Berlin, Springer 1992.

\bibitem{Reutenauer} 
	Ch.~Reutenauer, 
	{{Free Lie algebras}}.
	London Mathematical Society Monographs New Series 7. Oxford Science Publications 1993.

\bibitem{schutz}
	M.~P.~Sch\"utzenberger, 
	{\emph{Sur une propri\'et\'e combinatoire des alg\`ebres de Lie libres pouvant \^etre utilis\'ee 
	dans un probl\`eme de math\'ematiques appliqu\'ees}}, 
	S\'eminaire Dubreil--Jacotin Pisot (Alg\`ebre et th\'eorie des nombres), Paris, Ann\'ee 1958/59.

\bibitem{Strichartz}
    	R.~S.~Strichartz,
    	{\textsl{The Campbell--Baker--Hausdorff--Dynkin formula and solutions of differential equations}},
    	J.~Func.~Anal.~{\bf{72}}, 320 (1987). 
	

\end{thebibliography}
\normalsize
\end{document}